\documentclass[letterpaper,12pt]{amsart}
\title{Positivity for Gaussian Graphical Models}

\author{Jan Draisma}
\address{Department of Mathematics and Computer Science, Eindhoven
University of Technology; and Centrum voor Wiskunde en Informatica,
Amsterdam}
\thanks{JD was partially supported by a Vidi grant from the Netherlands
Organisation for Scientific Research (NWO).  SS was partially supported by the David and Lucille Packard Foundation
and the US National Science Foundation (DMS 0954865).  KT was partially supported by the US National Science Foundation (DMS 1004532)    }
\email{j.draisma@tue.nl}

\author{Seth Sullivant}
\address{Department of Mathematics \\
North Carolina State University, Raleigh, NC 27695}
\email{smsulli2@ncsu.edu}

\author{Kelli Talaska}
\address{Department of Mathematics \\
University of California, Berkeley, CA 94720}
\email{talaska@math.berkeley.edu}

\date{}

\usepackage{latexsym,array,delarray,amsthm,amssymb,epsfig}

\theoremstyle{plain}
\newtheorem{thm}{Theorem}[section]
\newtheorem{lemma}[thm]{Lemma}
\newtheorem{prop}[thm]{Proposition}
\newtheorem{cor}[thm]{Corollary}

\theoremstyle{definition}
\newtheorem{defn}[thm]{Definition}
\newtheorem{ex}[thm]{Example}

\theoremstyle{remark}

\newcommand{\ZZ}{\mathbb{Z}}

\newcommand{\rr}{\mathbb{R}}

\newcommand{\bfT}{\mathbf{T}}
\newcommand{\bfP}{\mathbf{P}}
\newcommand{\bfF}{\mathbf{F}}
\newcommand{\bfC}{\mathbf{C}}

\newcommand{\calf}{\mathcal{F}}

\newcommand{\caln}{\mathcal{N}}

\newcommand{\calt}{\mathcal{T}}

\newcommand{\sgn}{{\rm sign} \,}
\newcommand{\bi}{\leftrightarrow}

\newcommand{\tto}{\twoheadrightarrow}
\newcommand{\ott}{\twoheadleftarrow}
\newcommand{\frs}{\mathfrak{S}}
\newcommand{\UD}{\mathrm{UD}}
\newcommand{\opp}{\mathrm{opp}}

\newcommand{\ind}{\mbox{$\perp \kern-5.5pt \perp$}}
\newcommand{\indsub}{{\mbox{\scriptsize$\perp \kern-4.5pt \perp$}}}

\begin{document}

\begin{abstract}
Gaussian graphical models are parametric statistical models for jointly normal random variables whose dependence structure is determined by a graph.  In previous work, we introduced trek separation, which gives a necessary and sufficient condition in terms of the graph for when a subdeterminant is zero for all covariance matrices that belong to the Gaussian graphical model.  Here we extend this result to give explicit cancellation-free formulas for the expansions of nonzero subdeterminants.

\end{abstract}

\maketitle


\section{Introduction} \label{sec:Introduction}

Gaussian graphical models are parametric statistical models for jointly
normal random variables whose dependence structure is determined by a graph.
In this work, we consider Gaussian graphical models on mixed graphs
which have both directed edges and bidirected edges.  In the causal
modeling framework for graphical models \cite{Pearl2000}, the directed
edges represent direct causal effects of one variable on another, while
the bidirected edges represent the effects of unobserved confounders that
lead to correlations between the error terms of the variables.
In other contexts, the Gaussian graphical models we study are known
as linear structural equation models, and the directed edges describe linear
relationships between variables with correlated errors.

A fundamental problem in the study of graphical models is to characterize
the distributions that can arise.  The models are typically presented
using parametric descriptions, and we wish to give conditions on the
distributions which could appear for some choice of parameters.
In the case of graphical models which only have directed edges without directed cycles
 (so-called \emph{directed acyclic
graphs} or DAGs), it is well-known that a probability distribution
belongs to the model if and only if the distribution satisfies
all the conditional independence constraints implied by the graph
\cite[Thm 3.27]{Lauritzen1996}.

For Gaussian graphical models, the model consists of all covariance
matrices $\Sigma \in PD_{m}$ which arise for some choice of parameters.  Each
conditional independence statement $X_{A} \ind X_{B} | X_{C}$, translates
to a rank condition on a submatrix of the covariance matrix $\Sigma$,
namely that the matrix $\Sigma_{A \cup C, B \cup C}$ has rank $\leq \#C$.
(Note that for sets $I$ and $J$, $\Sigma_{I,J}$ denotes the submatrix of $\Sigma$
with row index set $I$ and column index set $J$.) In other words,
all $ (\#C +1) \times (\#C +1)$ subdeterminants of $\Sigma_{A \cup C, B \cup C}$
are zero.

In previous work \cite{Sullivant2010}, we generalized the
rank dropping condition from conditional independence to arbitrary minors,
via \emph{trek separation}.  This gives
necessary and sufficient conditions in terms of the underlying graph
$G$, and sets $A$ and $B$, such that $\det \Sigma_{A,B} = 0$. These
vanishing determinantal constraints are among the most natural to study.
In the present paper, we extend the result of \cite{Sullivant2010} to
give explicit cancellation-free expansions for the non-vanishing determinants
$\det \Sigma_{A,B}$.
This has potential applications to the algebraic study of graphical models
including to the identifiability problem for Gaussian
graphical models \cite{Drton2010, Foygel2011}.
Besides this, our formulas for $\det \Sigma_{A,B}$ involve elegant combinatorics
in the spirit of classical enumerative combinatorics results on determinants
of matrices associated to graphs \cite{Gessel1985,Lindstrom1973}.
For example, our results give the
expansion of $\det \Sigma_{A,B}$ in terms of  \emph{nonintersecting trek flows},
and we show that after collecting terms, every monomial in the expansion
appears with a coefficient of the form $\pm 2^{c}$ for some $c$, determined
by combinatorial properties of the flow.

This paper is organized as follows.  In the next section, we give
an introduction to Gaussian graphical models and the special ``paths" in
mixed graphs we need to study, called \emph{treks}.  In Section \ref{sec:Results}
we describe our main results on the expansions of the determinants
of submatrices of $\Sigma$ arising from a Gaussian graphical model.
These results are based on a generalization of the Lindstr\"om-Gessel-Viennot
Lemma for graphs with cycles, which is described in Section \ref{sec:Kelli}.
Finally, Section \ref{sec:Proofs} contains the proofs of the most general
statements from Section \ref{sec:Results} using the methods of Section \ref{sec:Kelli}.


\section{Gaussian Graphical Models} \label{sec:Gaussian}

Let $G = (V,B,D)$ be a mixed graph with vertex set $V$, bidirected edge
set $B$, and directed edge set $D$. Thus $(V,D)$ is a directed graph
without loops or multiple edges called the {\em directed part} of $G$,
and $(V,B)$ is an undirected graph without loops or multiple edges called
the {\em bidirected part} of $G$. The qualification {\em bidirected}
stems from these edges' statistical interpretation in graphical models.
A bidirected edge between vertices $i$ and $j$ is denoted $i \bi j$ and
a directed edge from $i$ to $j$ is denoted $i \to j$.    Let $PD_{V}$
denote the set of $\#V \times \#V$ real symmetric positive definite
matrices.  Let $PD(B)$ denote the set of positive definite matrices with off-diagonal nonzero
entries only in positions corresponding to bidirected edges of $G$, i.e.
$$
PD(B)  =  \{  \Omega = (\omega_{ij}) \in PD_{V} :  \omega_{ij} = 0 \mbox{
if } i \neq j \mbox{ and } i \bi j \notin B \}.
$$

Let $\epsilon \in \rr^{V}$ be a jointly normal random vector $\epsilon
\sim \caln(0, \Omega)$ with $\Omega \in PD(B)$.  For each $i \to j \in
D$, let $\lambda_{ij} \in \rr$ be a real parameter.  Define a new random
vector $X \in \rr^{V}$ by
\begin{equation} \label{eq:linear}
X_{j} =  \sum_{i : i \to j \in D} \lambda_{ij} X_{i}  + \epsilon_{j}.
\end{equation}
Let $\Lambda$ be the $\#V \times \#V$ matrix with $\lambda_{ij}$ in
the $ij$ position if $i \to j \in D$, and zero otherwise.  Equation
(\ref{eq:linear}) can be rearranged as $\epsilon = (I - \Lambda)^{T} X$,
where $I$ is a $\#V \times \#V$ identity matrix.  From this, we deduce
that $X \sim \caln(0, \Sigma)$ where
$$
\Sigma = (I- \Lambda)^{-T} \Omega (I - \Lambda)^{-1}
$$
provided that the matrix $I - \Lambda$ is invertible.

Given $A, B \subset V$, let $\Sigma_{A,B} = ( \sigma_{a,b})_{a \in
A, b \in B}$ be the submatrix of $\Sigma$ with row index set $A$
and column index set $B$.  Suppose $\#A = \#B$.  In previous work
\cite{Sullivant2010}, we gave a combinatorial condition for when
$\det \Sigma_{A,B} = 0$.  In this paper, we explain how to expand $\det
\Sigma_{A,B}$ explicitly in monomials without cancellation.  Henceforth,
rather than thinking about the $\lambda_{ij}$ and $\omega_{ij}$ as
parameters, we choose to think about them as indeterminates (or polynomial
variables).  Hence, we are interested in expanding $\det \Sigma_{A,B}$
as a polynomial or a rational function, depending on the context.

One of the key combinatorial definitions we need is the definition of
a trek.

\begin{defn}
A {\em trek} in $G$ from $i \in V$ to $j \in V$ is a pair $(P_L,P_R)$,
where $P_L$ is a directed path from some vertex $s$ to $i$ and $P_R$
is a directed path from some vertex $t$ to $j$, satisfying the further
requirement that if $s \neq t$, then there must be a bidirected edge
$s \bi t \in B$. The vertices $i$ and $j$ are called the {\em initial vertex}
and the {\em final vertex} of the trek, respectively.
\end{defn}

In this definition, a path is an ordered sequence of edges in $D$ in
which the head of each edge equals the tail of the next edge. A path may have no edges, in which case, we also specify where the path starts, which is also where it ends; we call such a path an empty path.  Later in the paper we will work mostly with
self-avoiding paths, in which each vertex is the head of at most one
edge of the path and the tail of at most one edge of the path; such
a path is uniquely determined by its edge set (and initial vertex
when empty). For the moment, however, $P_L$ and $P_R$ need not be
self-avoiding.  Note also that either (or both) of the paths $P_{L}$ and
$P_{R}$ are allowed to be an empty path, in which case $s=i$ or $t=j$,
respectively (or both). Two treks $T = (P_{L}, P_{R})$ and $T' = (P_L',
P_R')$ are the same if both $P_L=P_L'$ and $P_R=P_R'$. In particular,
two treks might contains exactly the same set of edges (even with the
same multiplicities) but be different treks.

If all directed edges in $G$ are pointing down the page---which they can
be arranged to do if the directed part of $G$ happens to be acyclic---then
one can think of a trek between $i$ and $j$ as a path of edges that
starts at $i$, goes up for a time to $s$, possibly traverses a single
bidirected edge to a vertex $t$, and then turns around and goes down for
a time to $j$; see Figure~\ref{fig:trekswap1} below. We stress, however,
that our results below are not restricted to the case where the directed
part of $G$ is acyclic.

To any trek $T=(P_L,P_R)$, we associate the {\em trek monomial} $m_{T}
= \lambda^{L}\omega_{st}\lambda^{R}$, where $\lambda^{L} =  \prod_{k
\to l \in P_L} \lambda_{kl}$, and similarly,  $\lambda^{R} =  \prod_{k
\to l \in P_R} \lambda_{kl}$.  Here $s,t$ are the initial vertices of
$P_L,P_R$, and a variable $\lambda_{kl}$ appears in $\lambda^L$ and $\lambda^R$
with its respective multiplicities in the paths $P_L$ and $P_R$.  The reason
for introducing treks and trek monomials is the following result:

\begin{prop}\label{prop:trekexpand}
Let $G = (V,B,D)$ be a graph, and $\Sigma = (I - \Lambda)^{-T} \Omega
(I - \Lambda)^{-1}$.  Then for all $i,j \in V$,
$$
\sigma_{ij}  = \sum_{T} m_{T},
$$
where the sum runs over all treks $T$ in $G$ from $i$ to $j$.
\end{prop}

Proposition \ref{prop:trekexpand} can be proven use the expansion
$$(I- \Lambda)^{-1} =  I + \Lambda + \Lambda^{2} + \cdots ,$$
and the combinatorial interpretation of the $kl$ entry of  $\Lambda^{r}$
as the sum of $\lambda^{P}$ over all directed paths $P$ of length $r$
from $k$ to $l$ in the graph $(V,D)$.  See \cite{Sullivant2010}.

\begin{figure}
\includegraphics[]{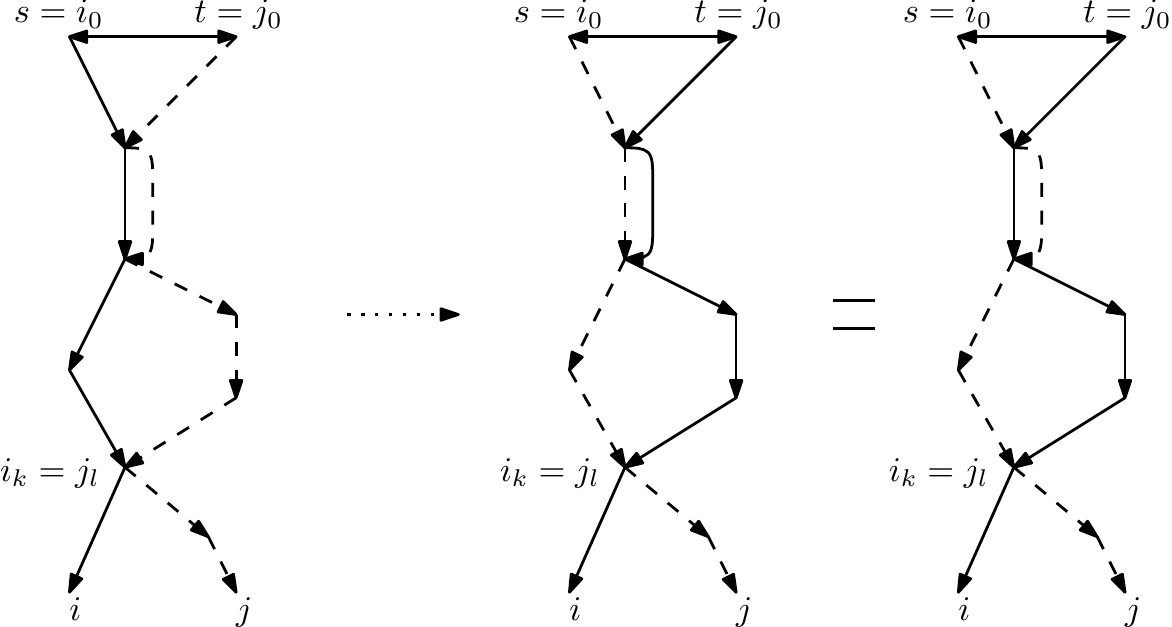}
\caption{Tailswapping at $(k,l)$ yields a trek with the same trek
monomial. The two treks on the right are the same. The number of treks from $i$ to $j$
with this trek system monomial equals
$2^{i(T)-e(T)}=2^{3-1}=4$.}
\label{fig:trekswap1}
\end{figure}

As stated previously, there might be many treks in $\calt(i,j)$ that
use exactly the same set of edges, so the expression for $\sigma_{ij}$
might have repeated terms.  As a prelude to
our main results, we discuss how to simplify this expression in the case
where $(V,D)$ is acyclic. Let $T=(P_L,P_R)$ be a trek with $P_L=(s=i_0
\to i_1 \to \ldots \to i_m=i)$ and $P_R=(t=j_0 \to j_1 \to \ldots \to
j_n=j)$. Since $(V,D)$ is acyclic, the $i_k$ are mutually distinct,
and so are the $j_l$. But it may happen that $i_k$ equals $j_l$, and
then $T':=(P_L',P_R')$ with $P_L':t=j_0 \to \ldots \to j_l \to i_{k+1}
\to \ldots \to i_m=i$ and $P_R':s=i_0 \to \ldots \to i_k \to j_{l+1}
\to \ldots \to j_n=j$ is a trek with $m_{T'}=m_T$. We call this procedure
{\em tailswapping} at $(k,l)$; see Figure~\ref{fig:trekswap1}. It turns
out that any trek $T'$ with $m_T=m_{T'}$ can be obtained from $T$ by
repeated tailswapping.  We call such treks {\em equivalent} to $T$,
and we write $T \sim T'$.  If $k=l=0$ (so that $s$ equals $t$), then
then clearly $T=T'$.  Similarly, if $k,l>1$ and $i_{k-1}=j_{l-1}$,
then tailswapping at $(k-1,l-1)$ gives the same result as tailswapping
at $(k,l)$. Denote by $i(T)$ the number of pairs $(k,l) \neq (0,0)$
with $i_k=j_l$ and by $e(T)$ the number of edges that $P_L$ and $P_R$
have in common.  The above discussion, together with the observation that
tailswapping at distinct pairs of indices commutes then yields the
following proposition.

\begin{prop} \label{prop:2n1}
Let $G = (V,B,D)$ be a mixed graph such that $(V,D)$ is directed acyclic,
and let $\Sigma = (I - \Lambda)^{-T} \Omega (I - \Lambda)^{-1}$. Then for
all $i,j \in V$,
$$
\sigma_{ij}  = \sum_{[T]_\sim} 2^{i(T)-e(T)}m_{T}.
$$
where the sum runs over equivalence classes of treks.
\end{prop}

We will greatly generalize Proposition \ref{prop:2n1} in Section
\ref{sec:Results}.  In particular, the fact that all coefficients in $\det
\Sigma_{A,B}$ are powers of $2$ persists when $(V,D)$ is a directed
acyclic graph.

Now we return to the general case, where $(V,D)$ may contain cycles. Since
each entry $\sigma_{ij}$ is the sum of trek monomials from $i$ to $j$,
the determinant $\det \Sigma_{A,B}$ will be a (signed) sum of products
of trek monomials.  Indeed, expanding the determinant using the standard
Leibniz expansion yields the following:

\begin{prop} \label{prop:DetSigmaEasy}
Let $A, B \subseteq V$, with $\#A = \#B = k$.  Let $A = \{a_{1}, \ldots, a_{k} \}$ and $B = \{b_{1}, \ldots, b_{k}\}$.  Then
$$
\det \Sigma_{A,B}  =  \sum_{\pi \in \frs_{k}} \left(\sum_{T_{1} \in \calt(a_{1}, b_{\pi(1)}), \ldots, T_{k} \in \calt(a_{k}, b_{\pi(k)})} \sgn(\pi) \cdot m_{T_{1}} \cdots m_{T_{k}}\right).
$$
\end{prop}

There are two issues with this expansion. The first is that distinct
tuples $(\pi,T_1,\ldots,T_k)$ can lead to the same monomial $m_{T_1}
\cdots m_{T_k}$ and thus cancellation can occur. The second is that when
the directed part of $G$ contains directed cycles, 
this expansion is a formal power
series rather than a polynomial (of course, this already happens when $A$
and $B$ are singletons). In the general case our aim is to write this
formal power series as a rational function in which the denominator and
the numerator are both in a suitable sense cancellation-free.

An important simplification that we can make when describing the graphs,
matrices, and determinants is to eliminate the bidirected edges under
consideration.  Associated to the graph $G = (V,B,D)$, we introduce
the bidirected subdivision $\tilde{G}$, which has only
directed edges.  We start will all vertices and directed edges in $G$.
Then, for each bidirected edge $i \bi j$ in $G$, we introduce a new
vertex $(i,j)$, and two directed edges $(i,j) \to i$ and $(i,j) \to j$.
The resulting graph is the bidirected subdivision of $G$.  Combinatorial
expressions for the entries $\sigma_{ij}$ associated to the graph $G$
can be obtained from the corresponding expressions for $\tilde{G}$ as
described in the following proposition.

\begin{prop}
Let $G = (V,B,D)$ be a graph and let $\tilde{G}$ be its bidirected
subdivision.  Write $\Sigma = (I - \Lambda)^{-T} \Omega (I - \Lambda)^{-1}$
for the covariance matrix for  the graph $G$ and
$\Sigma^{*} = (I - \Lambda^{*})^{-T} \Omega^{*} (I - \Lambda^{*})^{-1}$
for the covariance matrix for the bidirected subdivision $\tilde{G}$.
 Let $f \in \rr[\sigma]$ be a polynomial.
Then
$f(\sigma(\lambda, \omega)) = 0$ if and only if
$f(\sigma^{*}(\lambda^{*}, \omega^{*})) = 0$.

Furthermore,
the expansion of $f(\sigma(\lambda, \omega))$ can be recovered from the expansion
of $f(\sigma^{*}(\lambda^{*}, \omega^{*}))$ by the following procedure:
\begin{enumerate}
\item  Remove any monomial that contains a $(\lambda^{*}_{(i,j),i})^{2}$ or
$(\lambda^{*}_{(i,j),j})^{2}$.
\item Set all remaining $\lambda^{*}_{(i,j),i}$ and $\lambda^{*}_{(i,j),j}$
variables to $1$.
\item Change $\omega^{*}_{(i,j),(i,j)}$ to $\omega_{ij}$.
\item Change all remaining $\lambda^{*}_{ij}$ and $\omega^{*}_{ii}$ to
$\lambda_{ij}$ and $\omega_{ii}$, respectively.
\end{enumerate}
\end{prop}

\begin{proof}
Consider the $\rr$-algebra homomorphism
$\phi: \rr[[\lambda, \omega]] \rightarrow \rr[[\lambda^{*}, \omega^{*}]]$ defined by
$$ \phi(\lambda_{ij}) = \lambda^{*}_{ij}, \quad
\phi(\omega_{ii})  = \omega^{*}_{ii}  + \sum_{j: i \bi j \in B} (\lambda^{*}_{(i,j) i})^{2} \omega^{*}_{(i,j)(i,j)}, \quad
\phi(\omega_{ij}) = \lambda^{*}_{(i,j),i} \omega^{*}_{(i,j),(i,j)} \lambda^{*}_{(i,j),j}.
$$
We claim that $\sigma^{*}_{ij} = \phi(\sigma_{ij})$.

First, we prove the claim.  Let $\Lambda^{*}$ and $\Omega^{*}$ be the matrices for
the bidirected subdivision $\hat{G}$.  We realize both matrices as block matrices
with the first set of row/column labels corresponding to bidirected edges $i \bi j$,
and hence the vertices $(i,j)$ in $\hat{G}$.  Then $\Lambda^{*}$ and $\Omega^{*}$ have
the form
$$
\Lambda^{*} = \begin{pmatrix}
0  &  E  \\
0  &  \Lambda
\end{pmatrix}
 \quad \quad
\Omega^{*} =
\begin{pmatrix}
\Omega^{*}_{1} & 0 \\
0 & \Omega^{*}_{2}
\end{pmatrix}
$$
where $\Omega^{*}_{1} = diag( (\omega_{(i,j),(i,j)})_{i \bi j \in B} )$,
$\Omega^{*}_{2} = diag( (\omega_{ii})_{i \in V} )$, $\Lambda$ is the
directed edge matrix from
the  graph $G$, and $E$ is a matrix whose $(i,j), k$ entry is $\lambda_{(i,j),k}$ if
$k \in \{i,j\}$, and is zero otherwise.  Note that
$$(I - \Lambda^{*})^{-1}
=  \begin{pmatrix}
I &  (I-\Lambda)^{-1} E \\
0 & (I - \Lambda)^{-1}
\end{pmatrix}
$$
from which we deduce that
$$
(I - \Lambda^{*})^{-T} \Omega^{*} (I - \Lambda^{*})^{-1}  =
\begin{pmatrix}
\Omega^{*}_{1}  &  \Omega^{*}_{1} E (I- \Lambda)^{-1}  \\
(I- \Lambda)^{-T} E^{T} \Omega^{*}_{1}  &
(I- \Lambda)^{-T} (\Omega^{*}_{2} + E^{T} \Omega^{*}_{1} E) (I- \Lambda)^{-1}
\end{pmatrix}.
$$
The entry in the bottom right-hand corner of this matrix is the expression
for the submatrix $\Sigma^{*}_{V,V}$, so the claim is equivalent
to saying that the map $\phi$ is that map that takes $\Omega$ in
$(I- \Lambda)^{-T} \Omega (I- \Lambda)^{-1}$ and replaces it with
$\Gamma = \Omega^{*}_{2} + E^{T} \Omega^{*}_{1} E$.  But it is easy to see that
$\gamma_{ii} = \phi(\omega_{ii})$ and for $i \neq j$,
 $\gamma_{ij} = 0$ if $i \bi j \notin B$ and $\gamma_{ij} = \phi(\omega_{ij})$
 if $i \bi j \in B$.

Now we will show that the claim proves the proposition.
First of all, the map $\phi$ is injective.  Indeed, suppose that $\phi(f) = 0$ and
$f$ is irreducible.
Then, since $\omega^{*}_{ii}$ only appears in $\phi(\omega_{ii})$, then
$\omega_{ii}$ could not appear in $f$.  Similarly, we can then rule out
$\omega_{ij}$ appearing in $f$, and $\lambda_{ij}$ appearing in $f$.  This forces
$f$ to be the zero polynomial.  Since $\phi$ is injective, we deduce the first part
of the proposition, a power series  $f(\sigma(\lambda, \omega))$ is zero
if and only if $\phi(f) = f(\sigma^{*}(\lambda^{*}, \omega^{*}))$ is zero.

For the second part, we note that the following map $\psi$ applied to
$f(\sigma^{*}(\lambda^{*}, \omega^{*}))$ produces $f(\sigma(\lambda, \omega))$:
$$
\psi(\lambda^{*}_{ij}) = \lambda_{ij}, \quad  \psi(\omega^{*}_{ii}) = \omega_{ii}
-  \sum_{j: i \bi j \in B} \omega_{ij}, \quad  \psi(\omega^{*}_{(i,j), (i,j)} =
\omega_{ij}, \quad  \psi( \lambda^{*}_{(i,j),i} ) = 1.
$$
That can be seen by applying the operations on the factorization
$(I- \Lambda)^{-T} (\Omega^{*}_{2} + E^{T} \Omega^{*}_{1} E) (I- \Lambda)^{-1}$
to produce $(I- \Lambda)^{-T} \Omega (I- \Lambda)^{-1}$.  This corresponds
to the four-step procedure from the statement of the proposition,
since the expressions involving $(\lambda^{*}_{(i,j), i})^{2} \omega^{*}_{(i,j),(i,j)}$
will cancel with terms coming from $\psi(\omega^{*}_{ii}) = \omega_{ii}
-  \sum_{j: i \bi j \in B} \omega_{ij}$.
\end{proof}

Henceforth, we focus solely on the case of graphs with no bidirected edges.


\section{Results} \label{sec:Results}

In this section, we assume $G=(V,D)$ is a directed graph, possibly obtained as
the bidirected subdivision of some mixed graph. Our goal is to find a
cancellation-free analogue of the formula for $\det \Sigma_{A,B}$ in
Proposition~\ref{prop:DetSigmaEasy}. The exact meaning of
cancellation-free depends on the context. If $G$ is acyclic, then
$\det \Sigma_{A,B}$ is a polynomial in the variables $\omega_{ii}$
and the variables $\lambda_{ij}$ for directed edges $i \to j$, and
cancellation-free means that we determine which monomials have non-zero
coefficients in $\det \Sigma_{A,B}$, and what those coefficients are.

If, on the other hand, $G$ contains directed cycles, then there may be
infinitely many treks from $i$ to $j$. In this case $\det \Sigma_{A,B}$ is
a formal power series representing a rational function (as is clear from
$\Sigma=(I-\Lambda)^{-T}\Omega (I-\Lambda)^{-1}$ using Cramer's rule). We
will then write $\det \Sigma_{A,B}$ as a fraction $\frac{R}{S}$ where $R$
and $S$ are polynomials with known monomials and known coefficients.

The formula that we derive in the general case has $S=1$ when specialized
to the acyclic case, hence our proof will immediately focus on the
general case. For the purpose of exposition, however, we first present
the formula for the acyclic case, before dealing with the general case, which is more combinatorially complicated.

\subsection{Directed acyclic graphs}

\begin{defn}
Let $A$ and $B$ be sets of $k$ vertices.  A {\em trek system} $\bfT$ from $A$ to $B$ consists of $k$ treks whose initial vertices exhaust
the set $A$ and whose final vertices exhaust the set $B$.
\end{defn}

Such a trek system gives rise to a bijection $A \to B$ mapping $a$ to
the final vertex of the (unique) trek in $\bfT$ starting at $a$. Given
a linear ordering $a_1,\ldots,a_k$ of $A$ and a linear ordering
$b_1,\ldots,b_k$ of $B$, this bijection determines an element $\pi$
of $\frs_k$ (which depends on the
linear orderings of $A$ and $B$), and we define $\sgn \bfT:=\sgn \pi$. We define the {\em trek system monomial}
$m_\bfT$ as the the product of the trek monomials $m_T$ ranging over $T \in \bfT$. The
formula in Proposition~\ref{prop:DetSigmaEasy} can then be rewritten as
\begin{equation} \label{eq:DetSigmaEasy2}
\det \Sigma_{A,B}= \sum_{\bfT} \sgn(\bfT) m_{\bfT}
\end{equation}
where $\bfT$ runs over all trek systems from $A$ to $B$.  Cancellation can
happen if distinct $m_{\bfT}$ and $m_{\bfT'}$ are the same. To achieve
a cancellation-free formula we introduce the following notion (see
\cite{Sullivant2010}).

\begin{defn}
A trek system $\bfT$ from $A$ to $B$ has {\em no sided intersection} if
the left parts $P_L$ for $(P_L,P_R) \in \bfT$ are mutually vertex-disjoint
and also the right parts $P_R$ for $(P_L,P_R) \in \bfT$ are mutually
vertex-disjoint (but any $P_L$ may have vertices in common with
any $P_R'$). Otherwise, $\bfT$ is said to have a sided intersection.
We write $\calt(A,B)$ for the set of trek systems from $A$ to $B$ with
no sided intersection.
\end{defn}

In \cite{Sullivant2010} it is proved that in \eqref{eq:DetSigmaEasy2}
one may restrict $\bfT$ to run over all trek systems without sided
intersection. Our first new result is that these terms all have the same sign and thus do not cancel.

\begin{thm}[Positivity for acyclic digraphs]
\label{thm:PositivityAcyclic}
Let $G=(V,D)$ be an acyclic digraph and let $A$ and $B$ be subsets of $V$
of the same cardinality.  Assume we have fixed linear orderings of $A$ and $B$. If two
trek systems $\bfT$ and $\bfT'$ from $A$ to $B$ without sided intersection satisfy $m_\bfT = m_\bfT'$, then $\sgn(\bfT)=\sgn(\bfT')$.
\end{thm}

\begin{figure}
\includegraphics[]{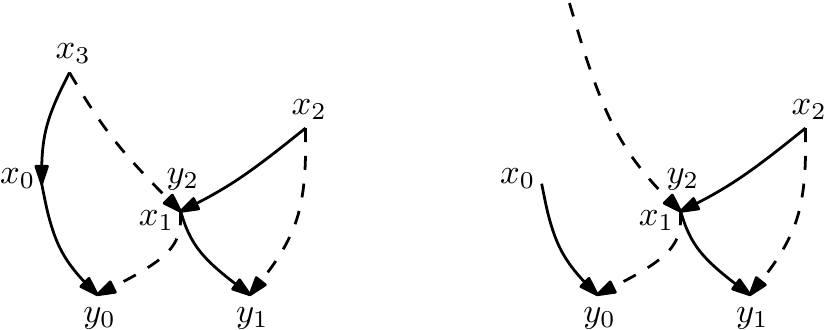}
\caption{Solid directed edges denote parts of left paths in $\bfT$, and dashed
directed edges denote parts of right paths in $\bfT$. On the left an up-down cycle
is created, on the right, we traverse a right path up without
hitting a further left path.}
\label{fig:updown1}
\end{figure}

To obtain the desired cancellation-free formula we therefore only need
to compute the number of trek systems $\bfT'$ with the same trek system monomial as
$\bfT$. For this we introduce combinatorial objects that we call {\em
up-down cycles} in $\bfT$.  Rather than giving a formal definition we
describe their construction; see 
Figure~\ref{fig:updown1} for an example. Let $x_0 \to y$ be an edge in some left path $P_L$
of $\bfT$, and assume that $x_0 \to y$ is {\em not} contained in any right
path of $\bfT$. Follow the directed path $P_L$ {\em downwards}, i.e. with the direction of the graph, starting
with $x_0 \to y$; since $\bfT$ has no sided intersection, we will never intersect
another left path. If no vertex on $P_L$ after the edge $x_0 \to y$ is in
any right path, then $x_0 \to y$ is not in any up-down cycle. Otherwise,
let $y_0$ be the first vertex on $P_L$ after $x_0 \to y$ (possibly $y_0$
equals $y$) that is also on some (necessarily unique) right path $P_R'$ of a trek $(P_L',P_R') \in \bfT$. Then follow $P_R'$ {\em upwards}, i.e. against the direction of the graph, from
$y_0$. If we never intersect another vertex on any left path (different from
$y_0$), then again $x_0 \to y$ is not in any up-down cycle. Otherwise,
let $x_1$ be the first vertex encountered when traversing $P_R'$ upwards
that is contained in any left path from $\bfT$; say $P_L''$.  Then follow
$P_L''$ down from $x_1$, etc. Continuing in this manner we construct a
sequence of non-empty paths $x_0 \tto y_0 \ott x_1 \tto \ldots$ in $G$
that are alternatingly contained in left paths and right paths of $\bfT$
and that do not contain edges shared by some left path with some right
path (because we always turn at the {\em first} possible vertex in each
step). Only three things can happen: we traverse a left
path downward to its final vertex without hitting a further right path,
we traverse a right path upward to its initial vertex
without hitting a further left path, or we eventually follow the
original path $P_L$ from somewhere above the edge $x_0 \to y$. In the
final case, the sequence of paths closes up and the union of their edge
sets is what we call an {\em up-down cycle} in $\bfT$. The set of all
up-down cycles in is denoted $\UD(\bfT)$. For a more formal definition
see the next subsection on the general directed case.

\begin{thm}[Power-of-two phenomenon for acyclic digraphs]
\label{thm:PowerOfTwoAcyclic}
Let $G=(V,D)$ be an acyclic digraph and let $A$ and $B$ be subsets of $V$
of the same cardinality, and assume we have fixed linear orderings of $A$ and $B$, as in
Theorem~\ref{thm:PositivityAcyclic}. Let $\bfT$ be a
trek system in $G$ without sided intersection. Then the number of trek systems $\bfT'
\in \calt(A,B)$ with $m_{\bfT'}=m_{\bfT}$ equals $2^{|\UD(\bfT)|}$.
\end{thm}

A consequence of the previous two theorems is the following formula.

\begin{cor}[Cancellation-free formula for acyclic digraphs]
\label{cor:CancellationFreeAcyclic}
Let $G=(V,D)$, $A$, and $B$ be as in Theorem~\ref{thm:PositivityAcyclic}. Then
\[ \det \Sigma_{A,B} = \sum_{[\bfT]_\sim \in \calt(A,B)/\sim} \sgn(\bfT) 2^{|\UD(\bfT)|} m_{\bfT} \]
where the sum runs over equivalence classes of the relation $\sim$
defined by $\bfT \sim \bfT'$ if and only if $m_{\bfT}=m_{\bfT'}$.
\end{cor}

\subsection{General directed graphs} \label{ssec:General}

For a general digraph $G=(V,D)$ that may contain cycles, the analogues
of Theorems~\ref{thm:PositivityAcyclic} and~\ref{thm:PowerOfTwoAcyclic}
and Corollary~\ref{cor:CancellationFreeAcyclic} require some notions
from \cite{Talaska12}.

\begin{defn}
Let $A=\{a_1,\ldots,a_k\}$ and $B=\{b_1,\ldots,b_k\}$ be $k$-subsets of $V$ with fixed linear orderings.  A {\em self-avoiding flow}
from $A$ to $B$ in $G$ is a pair $\bfF=(\bfP,\bfC)$ where $\bfP$ is
a set of $k$ self-avoiding and pairwise vertex-disjoint paths whose
initial vertices exhaust $A$ and whose final vertices exhaust $B$, and
where $\bfC$ is a set of non-empty, self-avoiding directed cycles that
are pairwise vertex-disjoint and that are also vertex-disjoint from all
paths in $\bfP$. The path component $\bfP$ gives rise to a bijection
$A \to B$, which through the linear orderings corresponds to an element
$\pi \in \frs_k$.  The {\em sign} $\sgn \bfF$ is defined as $\sgn(\pi)
\cdot (-1)^{|\bfC|}$, where $|\bfC|$ denotes the number of cycles in the collection $\bfC$.
\end{defn}

In this definition a self-avoiding cycle is identified with its set of
edges (without distinguished initial vertex). We introduce a trek analogue of this notion
as follows.

\begin{defn}
Let $A=\{a_1,\ldots,a_k\}$ and $B=\{b_1,\ldots,b_k\}$ be $k$-subsets of $V$ with fixed linear orderings.  A {\em self-avoiding trek
flow} in the digraph $G$ is a pair $\bfT=(\bfF_L,\bfF_R)$ where $\bfF_L$
is a self-avoiding flow from a (necessarily unique) $k$-subset $S$ of $V$ to $A$,
and $\bfF_R$ is a self-avoiding flow from the same set $S$ to $B$. We
call $S$ the set of {\em tops} of the trek flow. We write $\calt(A,B)$
for the set of self-avoiding trek flows from $A$ to $B$.
\end{defn}

Note that our notation $\bfT$ and $\calt(A,B)$ is consistent with the acyclic
case: there the cycle components $\bfC_L$ and $\bfC_R$ of $\bfF_L$ and $\bfF_R$
are empty, and the condition that the path component of $\bfF_L$
(respectively, $\bfF_R$) consists of self-avoiding and pairwise
vertex-disjoint paths is equivalent to the condition that a trek system
has no sided intersection.

\begin{defn}
The {\em sign} of a self-avoiding trek flow
$\bfT=(\bfF_L,\bfF_R)$ from $A$ to $B$ (relative to fixed linear
orderings on these sets) is defined as $\sgn(\bfT)=\sgn \bfF_R \cdot \sgn \bfF_L$.
(This depends on the orderings of $A$ and $B$, but not on the linear
ordering of the set of tops, as long as the same ordering is used for
determining the signs of both $\bfF_R$ and $\bfF_L$).

The {\em trek flow monomial} $m_\bfT$ is defined as the product of
the variables $\omega_{ii}$ for $i\in S$,
the variables $\lambda_{ij}$ corresponding to all edges $i \to j$ used
by paths and cycles in $\bfF_L$ and $\bfF_R$, taken with the appropriate multiplicities.
\end{defn}

Note that, in particular, the fact that $\bfT$ is self-avoiding implies
that each $\lambda_{ij}$ appears with degree at most $2$ in $m_\bfT$.
Using a theorem by the third author \cite{Talaska12}, we will write $\det
\Sigma_{A,B}$ as a rational function where both the numerator and the
denominator are signed sums over self-avoiding trek flows; see below. The
following theorem says that no cancellation occurs.

\begin{thm}[Positivity for general digraphs]
\label{thm:PositivityGeneral}
Let $G=(V,D)$ be a digraph and let $A$ and $B$ be subsets of $V$ of the same
cardinality, both with fixed linear orderings. If two self-avoiding
trek flows $\bfT$ and $\bfT'$ from $A$ to $B$ satisfy $m_\bfT = m_\bfT'$, then $\sgn(\bfT)=\sgn(\bfT')$.
\end{thm}

Next we compute the number of trek flows $\bfT'$ with the same trek flow
monomial as $\bfT$. Again, our main combinatorial tools are up-down cycles, which we define more formally at this point.

\begin{defn}
Let $\bfT=(\bfF_L,\bfF_R)$ be a self-avoiding trek flow, and let $E$ be
the set of edges used by $\bfF_L$ or by $\bfF_R$ but {\em not} by both.
Define a directed graph $\Gamma$ with vertex set
$E$ as follows: an element $e=x \to y \in E$ used by $\bfF_L$ has at
most one outgoing arrow in $\Gamma$. If $y$ is visited by some path
or cycle in $\bfF_R$ {\em and} has an incoming edge $f=z \to y$ there,
then the arrow from $e$ points to $f$. If $y$ is not visited by any path
or cycle in $\bfF_R$ (not even by an empty path based at $y$) {\em and}
has an outgoing edge $f=y \to z$ in $\bfF_L$, then the arrow from $e$
points to $f$. In all other cases, $e$ has no outgoing arrow in $\Gamma$.
Similarly, an element $e=x \to y \in E$ used by $\bfF_R$ has at most one
outgoing arrow. If $x$ is visited by $\bfF_L$ {\em and} has an outgoing
edge $f=x \to z$ there, then the arrow from $e$ points to $f$. If $x$
is not visited by $\bfF_L$ {\em and} has an incoming arrow $f=z\to x$
in $\bfF_R$, then the arrow from $e$ points to $f$. In all other cases,
$e$ has no outgoing arrow.

Now an {\em up-down cycle} in $\bfT$ is the support $F \subseteq E
\subseteq D$ of any non-empty directed cycle in the digraph $\Gamma$. We
write $\UD(\bfT)$ for the set of all up-down cycles in $\bfT$.
\end{defn}

In fact, the definition of $\Gamma$ forces every element of $E$
to have at most one incoming arrow, as well. As a consequence, $\Gamma$
is a union of vertex-disjoint directed cycles and paths. This implies
that distinct up-down cycles in $\bfT$ do not share edges, a fact that
we will use later. It is easily verified that this definition of up-down
cycle agrees with the construction given in the acyclic case. One new
aspect is that the edges of a directed cycle in the cycle component
of $\bfF_L$ which avoid $\bfF_R$ (or vice versa) also form an up-down cycle.

\begin{thm}[Power-of-two phenomenon for general digraphs]
\label{thm:PowerOfTwoGeneral}

Let $G=(V,D)$ be a digraph and let $A$ and $B$ be subsets of $V$ of the same
cardinality, both with fixed linear orderings, as in Theorem~\ref{thm:PositivityGeneral}. Let $\bfT$ be a
self-avoiding trek flow in $G$. Then the number of trek flows $\bfT' \in \calt(A,B)$
with $m_{\bfT'}=m_{\bfT}$ equals $2^{|\UD(\bfT)|}$.
\end{thm}

Finally, we state our general cancellation-free rational expression for
$\det \Sigma_{A,B}$.

\begin{cor}[Cancellation-free formula for general digraphs]
\label{cor:CancellationFreeGeneral}
Let $G=(V,D)$, $A$, and $B$ be as in Theorem~\ref{thm:PositivityGeneral}. Then
\[ \det \Sigma_{A,B} =
\frac{\sum_{[\bfT]_\sim \in \calt(A,B)/\sim} \sgn(\bfT) 2^{|\UD(\bfT)|} m_{\bfT}}{
\sum_{[\bfT]_\sim \in \calt(\emptyset,\emptyset)/\sim} \sgn(\bfT) 2^{|\UD(\bfT)|}
m_{\bfT}},
\]
where both run over equivalence classes of the relation $\sim$ defined by
$\bfT \sim \bfT'$ if and only if $m_{\bfT}=m_{\bfT'}$.
\end{cor}

Note that every self-avoiding trek flow from the empty set to itself is
of the form $((\emptyset,\bfC_L),(\emptyset,\bfC_R))$. In particular,
when $G$ is acyclic and there are no cycles, the denominator contains
only one term, corresponding to $\bfC_L=\bfC_R=\emptyset$, and our formula
specializes to the formula in Corollary~\ref{cor:CancellationFreeAcyclic}.


\section{Determinants of path matrices for cyclic graphs}\label{sec:Kelli}

In this section we recall a generalization of the
Lindstr\"om-Gessel-Viennot Lemma \cite{Lindstrom1973} to arbitrary graphs;
for a different
generalization see \cite{Fomin2001}. Thus let $G=(V,D)$ be an arbitrary finite
directed graph, and associate an edge variable $\lambda_{ij}$ to each
$i \to j \in D$. Let $A=\{a_1,\ldots,a_k\}$ and $B=\{b_1,\ldots,b_k\}$
be $k$-subsets of $V$, and let $M=M_{A,B}$ denote their {\em weighted path
matrix}, i.e., the $k \times k$-matrix over $\rr[[\lambda_{ij} \mid i \to
j \in D]]$ whose $(i,j)$-entry is the sum over all directed paths in $G$
from $a_i$ to $b_j$ of the product of the edge variables along the path.

To express $\det(M)$ as a rational function in the edge variables, we
recall from Section~\ref{sec:Results} the notion of self-avoiding flow
from $A$ to $B$ in $G$, and we write $\calf_{A,B}(G)$ for the (necessarily finite) set of all such flows.

\begin{thm}[\cite{Lalonde1996,Talaska12}]
\label{thm:PathMatrixDet}
The determinant of the weighted path matrix from $A$ to $B$
is given by
\[ \det(M_{A,B})=
\frac{\sum_{\bfF \in \calf_{A,B}(G)} \sgn(\bfF) m_{\bfF}}
{\sum_{\bfC \in \calf_{\emptyset,\emptyset}(G)}
\sgn(\bfC) m_{\bfC}},
\]
where the denominator is the sum over all self-avoiding flows consisting
of vertex-disjoint cycles only.
\end{thm}


\section{Proofs of the main results}\label{sec:Proofs}

In this section we prove Theorems~\ref{thm:PositivityGeneral} and
\ref{thm:PowerOfTwoGeneral} and
Corollary~\ref{cor:CancellationFreeGeneral}, which immediately imply
their acyclic special cases.  Let $G=(V,D)$ be a directed graph and
let $A=\{a_1,\ldots,a_k\}$ and $B=\{b_1,\ldots,b_k\}$ be $k$-subsets of
$V$. We first deduce the corollary from the two theorems.

\begin{proof}[Proof of Corollary~\ref{cor:CancellationFreeGeneral} given
Theorems~\ref{thm:PositivityGeneral} and \ref{thm:PowerOfTwoGeneral}]
We will apply Theorem~\ref{thm:PathMatrixDet} with $G$ replaced by the
digraph $H$ constructed from $G$ as follows: in the disjoint union of the
opposite graph $G^\opp$ with $G$ add an arrow from any vertex $i^\opp$ of
$G^\opp$ to the corresponding vertex $i$ in $G$. These arrows get labels
$\omega_{ii}$, while every edge $j^\opp \to i^\opp$ in $G^{\opp}$ gets the
same label $\lambda_{ij}$ as its opposite $i \to j$ in $G$. With this notation,
the submatrix $\Sigma_{A,B}$ of our graphical model equals the path matrix
$M_{A^\opp,B}$ in Theorem~\ref{thm:PathMatrixDet}. Similarly,
self-avoiding trek flows in $G$ from $A$ to $B$ are in
 bijection with self-avoiding flows in $H$
from $A^\opp$ to $B$, and the bijection preserves signs. Thus Theorem~\ref{thm:PathMatrixDet} yields the
expression
\[ \det \Sigma_{A,B} =
\frac{\sum_{\bfT \in \calt(A,B)} \sgn(\bfT) m_{\bfT}}{
\sum_{\bfT \in \calt(\emptyset,\emptyset)} \sgn(\bfT) m_{\bfT}}.
\]
Theorem~\ref{thm:PositivityGeneral} shows that self-avoiding
trek flows from $A$ to $B$ with the same trek flow monomial have the
same sign, and Theorem~\ref{thm:PowerOfTwoGeneral} counts the number of
such trek flows in each equivalence class. This gives the desired formula
\[ \det \Sigma_{A,B} =
\frac{\sum_{[\bfT]_\sim \in \calt(A,B)/\sim} \sgn(\bfT)
2^{|\UD(\bfT)|} m_{\bfT}}{
\sum_{[\bfT]_\sim \in \calt(\emptyset,\emptyset)/\sim} \sgn(\bfT)
2^{|\UD(\bfT)|}
m_{\bfT}}.
\]
\end{proof}

The remainder of this section focuses on the proofs of
Theorems~\ref{thm:PositivityGeneral} and~\ref{thm:PowerOfTwoGeneral}. The
proofs are closely intertwined and mostly contained in the main body of
the text.

First we want to reduce our arguments to the case where $A$ and $B$ are
disjoint. For this, we first modify $G$ as follows. For each $a_i,b_i$,
introduce copies $a_i',b_i'$ with edges $a_i \to a_i',b_i \to b_i'$
and call the resulting graph $G'$.  Let $A'=\{a_1',\ldots,a_k'\}$
and $B'=\{b_1',\ldots,b_k'\}$ inherit their linear orderings from $A$ and $B$ respectively. Adding the
new edges gives a bijection from self-avoiding trek flows in $G$ from
$A$ to $B$ to self-avoiding trek flows in $G'$ from $A'$ to $B'$. For
Theorem~\ref{thm:PositivityGeneral} we observe that this bijection
preserves signs, and for Theorem~\ref{thm:PowerOfTwoGeneral} we observe
that it also preserves the number of up-down cycles---indeed, the new
edges are not part of any up-down cycle. Thus, we may drop the accents and write
$G,A,B,a_i,b_i$ for $G',A',B',a_i',b_i'$, or in other words, assume $A$ and $B$ are disjoint subsets of $V$.

Next fix a self-avoiding trek flow
$\bfT=(\bfF_L=(\bfP_L,\bfC_L),\bfF_R=(\bfP_R,\bfC_R))$ from $A$ to $B$
with set of tops $S=\{s_1,\ldots,s_k\}$, and choose the linear orderings
such that $\bfP_L$ and $\bfP_R$ connect $s_i,\ i \in [k]$ to $a_i$ and to $b_i$,
respectively, so that the path components of both flows have sign $1$. We
want to show that any self-avoiding trek flow $\bfT'$ from $A$ to $B$
with the same trek flow monomial as $\bfT$ has the same sign as $\bfT$,
and that the number of such self-avoiding trek flows is $2^{|\UD(\bfT)|}$.

To this end, color the edges of $G$ with $0$ if they are in some path or
cycle of $\bfF_L$ and with $1$ if they are used by $\bfF_R$. Delete edges
of $G$ without any color; these are irrelevant for what follows. Then
contract all bi-colored edges in $G$. This may yield some additional
isolated vertices (corresponding to fully contracted cycles all of
whose edges are bi-colored); delete all isolated vertices. Contracting
bi-colored edges and deleting resulting empty cycles and other isolated
vertices gives a bijection between self-avoiding trek flows $\bfT'$ from
$A$ to $B$ in the original graph satisfying $m_{\bfT'}=m_{\bfT}$ and
self-avoiding trek flows in the resulting graph having {\em full support}
(i.e., containing all edges exactly once) and having $S$ as set of
tops. This bijection is sign-preserving because it does not affect the
bijections $A \to B$ and because it deletes cycles in pairs. Also note
that the number of up-down cycles in $\bfT$ does not change, since
up-down cycles by definition do not contain bi-colored edges. Again,
we change notation so that $G$ stands for the new graph.

Next, consider any monochromatic cycle in $G$ that has no vertices in
common with paths or cycles of the other color. This cycle can be moved
freely between the left and right cycle components of a self-avoiding
trek flow with full support without changing the sign or the trek flow monomial.
Moving such a cycle gives a factor $2$ when counting self-avoiding trek flows of
full support with $S$ as set of tops, and this factor is accounted for
by the fact that the cycle counts as an up-down cycle. So it suffices
to prove our theorems for the simplified graph where all such cycles
(edges and vertices) are deleted from $G$. Update $G$ again.

Finally, contract all vertices with exactly one incoming edge and one
outgoing edge (necessarily of the same color). This can now be done
without contracting a monochromatic cycle (since such cycles always
intersect a cycle of the other color after we have removed the
monochromatic cycles from the preceding according to the preceding
paragraph). Again, there is a sign-preserving
bijection, and the number of up-down cycles is invariant.

After these modifications, the vertices in $S$ have exactly two outgoing
edges, one colored $0$ and one colored $1$, and no incoming edges,
the vertices in $A \cup B$ have exactly one incoming edge colored $0$
(for $A$) or $1$ (for $B$) and no outgoing edges, and all other vertices
in $G$ have exactly two incoming edges (of different colors) and exactly
two outgoing edges (again of different colors). The resulting graph,
again denoted $G=(V,D)$, does not have any loops (but of course it may
have directed cycles).

We can recover $\bfT$ from the coloring by taking
for $\bfF_R$ the $0$-colored edges and for $\bfF_L$ the $1$-colored
edges. We may therefore identify $\bfT$ with the $\ZZ/2\ZZ$-coloring.
More generally, self-avoiding trek flows $\bfT'$ from $A$ to $B$ with the
same monomial as $\bfT$ correspond bijectively to edge colorings $D \to
\ZZ/2\ZZ$ with the same properties as $\bfT$, namely: the incoming edges
into $A$ are colored $0$, the incoming edges into $B$ are colored $1$,
the outgoing edges of each vertex in $V - (A \cup B)$ are colored with
distinct colors, and so are the two incoming edges of each vertex in $V
\setminus (S \cup A \cup B)$. We identify such $\bfT'$ with their
colorings. Then
the sum $\bfT'+\bfT$ (modulo $2$) is an edge coloring where the sum of
the colors of the edges entering any vertex is zero, and so is the sum
of the colors leaving any vertex. Conversely, adding a coloring $h$
with these properties to $\bfT$ yields a self-avoiding trek flow $\bfT'$
with the same monomial as $\bfT$.  Since the conditions on $h$ define
a subgroup $K(\bfT)$ of $(\ZZ/2\ZZ)^D$ the number of self-avoiding trek flows
$\bfT'$ with the same monomial as $\bfT$ is a power of $2$. The following
proposition completes the proof of Theorem~\ref{thm:PowerOfTwoGeneral}.

\begin{prop}
To each up-down cycle $C$ in $\bfT$ assign an edge coloring $D \to
\ZZ/2\ZZ$ that colors the edges in $C$ with $1$ and the remaining edges
with $0$. The resulting {\em up-down-colorings} form a basis of $K(\bfT)$
as a vector space over $\ZZ/2\ZZ$.
\end{prop}

\begin{proof}
Let $h$ be a non-zero element of $K(\bfT)$ and let $x_0 \to y_0$ be an edge
contained in the left part of $\bfT$ that is colored $1$ by $h$. Then
$y_0$ has a unique other incoming edge $x_1 \to y_0$ that must also
be colored $1$ by $h$, and this edge must be in the right part of
$\bfT$. Then $x_1$ has a unique other outgoing edge $x_1 \to y_2$ that
must also be colored $1$ by $h$, and this edge must be in $\bfT_L$, etc.
All these edges are distinct, until you get back to $x_0 \to y_l$. This
gives an up-down cycle whose coloring can be subtracted from $h$
to get a coloring with fewer edges colored $1$. This proves that
up-down colorings span $K(\bfT)$. Since distinct up-down cycles are
edge-disjoint, up-down colorings are also linearly independent.
\end{proof}

To prove that the signs of $\bfT'$ and $\bfT$ are the same, it now
suffices to prove that adding to $\bfT$ a coloring supported on an
up-down cycle does not change the sign. For this we prove Lemma
\ref{lemma:signswitch}.
To do this we first need to give a  definition of signs
of other combinatorial objects, which are used in the proof.

\begin{defn} \label{def:sgndeltaupsilon}
Let $\delta$ be an element of $\frs_k$, let $X,Y$ be subsets of $[k]$ of
the same cardinality, and let $\upsilon$ be a bijection from $[k]-Y$
to $[k]-X$. We define the {\em sign} of the pair $(\delta,\upsilon)$
as follows. Construct a directed bipartite graph $H$ on two copies
$[k]_1,[k]_2$ of $[k]$ with $\delta$ prescribing the arrows from $[k]_1$
to $[k]_2$ (``down'') and $\upsilon$ prescribing the arrows from $[k]_2$
(``up''). The graph $H$ defines a bijection $\pi: X \to Y$ which maps
$x \in X$ to the endpoint of the path in $H$ starting at $X$.  The sign
$\sgn(\delta,\upsilon)$ is defined as the sign of $\pi$ (relative to the
natural linear orderings on $X,Y \subseteq [k]$) times $(-1)$
raised to the number of directed cycles in $H$.
\end{defn}

\begin{lemma}\label{lemma:signswitch}
Let $\upsilon$ be as in the definition. Then
$\sgn(\delta,\upsilon)\sgn(\delta)$ depends only on $\upsilon$ and not
on $\delta \in \frs_k$.
\end{lemma}

\begin{proof}
We prove the lemma by induction on $k-|Y|(=k-|X|)$. It is
trivially true if $X$ and $Y$ are both equal to $[k]$: then the
graph $H$ does not have any upward arrows and $\pi=\delta$,
and $\sgn(\delta,\upsilon)=\sgn(\pi)=\sgn(\delta)$, so that
$\sgn(\delta,\upsilon)/\sgn(\delta)$ equals $1$, independently of
$\delta$.

Now suppose that the lemma is true for all $\upsilon: [k]-Y \to [k]-X$.
Pick such an $\upsilon$, let $x \in X$ and $y \in Y$, and let $\upsilon'$
be the extension of $\upsilon$ to $[k]-Y+y$ sending $y$ to $x$. Let
$\delta \in \frs_k$ and construct $H$ and $\pi$ as in the
definition (from
the pair $\delta,\upsilon$).  To analyze how $\sgn(\delta,\upsilon')$
differs from $\sgn(\delta,\upsilon)$ we distinguish two cases:
\begin{enumerate}
\item $\pi(x)=y$: In this case, adding the upward arrow $y \to x$ to $H$
creates an additional cycle in $H$, and this contributes a factor $-1$
to the sign.  Furthermore, the new bijection $\pi':X-x \to Y-y$ is the
restriction of $\pi$ to $X-x$. Hence we have
\[ \sgn(\pi')/\sgn(\pi)=
\sgn\left(\prod_{i \in X-x} (i-x)(\pi(i)-y)\right)=
(-1)^{|\{i \in X \mid i<x\}|+|\{j \in Y \mid j<y\}|}. \]
In conclusion, $\sgn(\delta,\upsilon')/\sgn(\delta,\upsilon)$ is
this latter expression times $-1$ (for the additional cycle
created in $H$).
\item $\pi(x)=z \neq y$: In this case, no new cycle is created when
adding to $H$ the upward arrow $y \to x$. Let $u:=\pi^{-1}(y)$, so that
in the new graph the path starting at $u \in X$ passes through $u,y,x,z$
in that order. Then the new permutation $\pi':X-x \to Y-y$ equals $\pi$
on $X-u-x$ and sends $u$ to $z$. Hence we have
\begin{align*}
\frac{\sgn(\pi')}{\sgn(\pi)}&=
\sgn \left(
\frac{\prod_{i \in X-u-x} (i-u)(\pi(i)-z)}
{(u-x)(y-z)\prod_{i \in X-u-x}[(i-u)(\pi(i)-y)(i-x)(\pi(i)-z)]}
\right)\\
&=
\sgn \left(
\frac{1}
{(u-x)(y-z)\prod_{i \in X-u-x}[(\pi(i)-y)(i-x)]}
\right)\\
&=
-\sgn \left(
\frac{1}
{\prod_{i \in X-x}[(i-x)]
\prod_{i \in X-u}[(\pi(i)-y)]}
\right)\\
&=-(-1)^{|\{i \in X \mid i<x\}|+|\{j \in Y \mid j<y\}|.
}
\end{align*}
\end{enumerate}
We conclude that the change of sign is the same in both cases, and
that it does not depend on $\delta$.
\end{proof}

We will use the following direct consequence of the lemma.

\begin{lemma} \label{lm:delta01}
Fix a natural number $k$. Let $\delta_0,\delta_1$ be
elements of $\frs_k$
and let $\upsilon_0,\upsilon_1$ be bijections between
subsets $[k]-Y_0,[k]-Y_1$ and $[k]-X_0,[k]-X_1$, respectively. Hence
$|X_i|=|Y_i|$ for $i=0,1$; we do not require that
$|X_0|=|X_1|$. Then we have the following identity:
\[
\sgn(\delta_0,\upsilon_0) \cdot \sgn(\delta_1,\upsilon_1)=
\sgn(\delta_1,\upsilon_0) \cdot \sgn(\delta_0,\upsilon_1). \]
\end{lemma}

Now we return to our proof of positivity; the following arguments are
clarified in Example~\ref{ex:updownswap} below. Suppose that $\bfT'$ and
$\bfT$ differ by an up-down cycle $h$, considered as an edge coloring
as above. Label the vertices with incoming edges of $h$-color $0$ and
outgoing edges of $h$-color $1$ by the numbers $1,\ldots,k$ (we identify
these with a first copy $[k]_1$ of $[k]$), and similarly label the
vertices with outgoing edges of $h$-color $0$ and incoming $h$-color $1$
by $1,\ldots,k$ (identifying them with a second copy $[k]_2$). Note that
there are, indeed, equally many of both, since the edges of $h$-color
$1$ and $\bfT$-color $0$ form $k$ vertex-disjoint paths from the first
set to the second set. Denote the corresponding bijection by $\delta_0
\in \frs_k$. Similarly, the edges of $h$-color $1$ and $\bfT$-color
$1$ give a bijection $\delta_1 \in \frs_k$. Now some directed paths with
$\bfT$-color $0$ emanating from $[k]_2$ (hence with $h$-color $0$) may
``loop back'' to $[k]_1$; this gives an bijection $\upsilon_0$ from a
subset $[k]_2-Y_0$ of $[k]_2$ to a $[k]_1-X_0$. Similarly, we obtain an
injective map $\upsilon_1$ from a subset of $[k]_2-Y_1$ into $[k]_1-X_1$
by following directed paths of $\bfT$-color $1$ emanating from $[k]_2$.

To change $\bfT$ into $\bfT'$ the colors in the up-down cycle are
interchanged. This means that the roles of $\delta_0$ and $\delta_1$
are interchanged. The total change in trek flow sign is exactly the
change from $\sgn(\delta_0,\upsilon_0) \sgn(\delta_1,\upsilon_1)$
into $\sgn(\delta_1,\upsilon_0) \sgn(\delta_0,\upsilon_1)$.
Indeed, the change in sign of the left flow $\bfF_L$ is
exactly $\sgn(\delta_1,\upsilon_0)/\sgn(\delta_0,\upsilon_0)$,
and the change in sign of the right flow $\bfF_R$ is exactly
$\sgn(\delta_0,\upsilon_1)/\sgn(\delta_1,\upsilon_1)$. Now
Lemma~\ref{lm:delta01} implies that that $\bfT$ and $\bfT'$ have the same
sign. This concludes our proof of Positivity for general digraphs. We
conclude the paper with an example illustrating the arguments just given.

\begin{ex} \label{ex:updownswap}
\begin{figure}
\includegraphics[width=\textwidth]{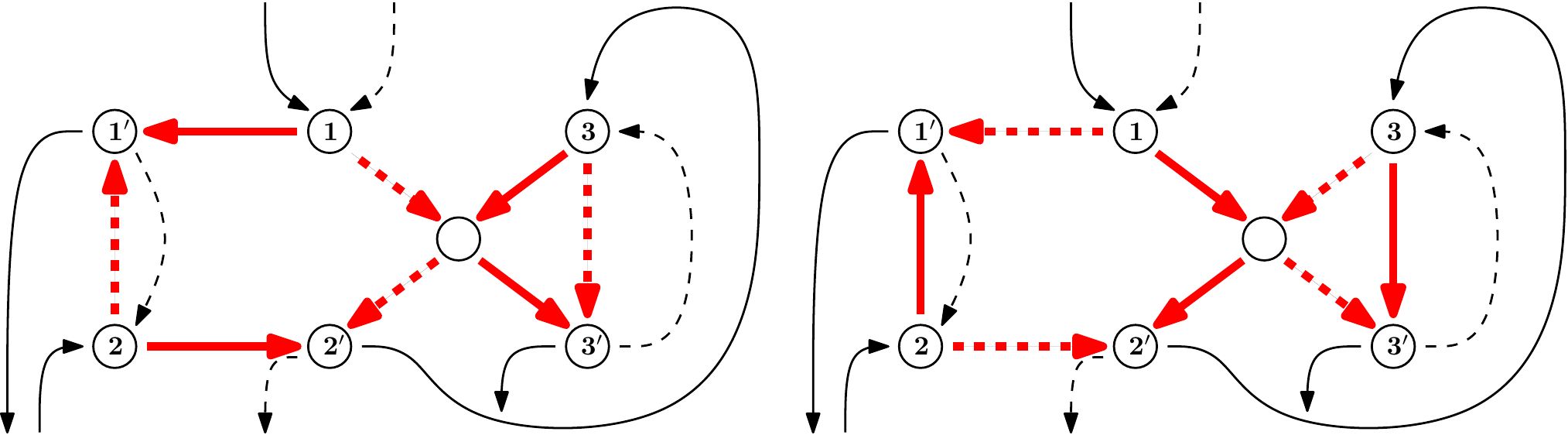}
\caption{Left: a fragment of $\bfT$; and right: the corresponding fragment
of $\bfT'$.}
\label{fig:examplepos}
\end{figure}

In Figure~\ref{fig:examplepos} we have $k=3$, the elements of $[3]_1$
are denoted $1,2,3$, and the elements of $[3]_2$ are denoted $1',2',3'$.
On the left is a fragment of $\bfT$ with $0$-colored edges solid and
$1$-colored edges dashed. On the right is the corresponding fragment
of $\bfT'$.

The bijection $\delta_0$ maps $1 \mapsto 1',2 \mapsto 2',3 \mapsto 3'$
and is depicted by straight solid arrows on the left and by dashed
solid arrows on the right. The bijection $\delta_1$ maps $1 \mapsto 2',2 \mapsto 1',
3 \mapsto 3'$ and is depicted by straight dashed arrows on
the left and by straight solid arrows on the right. Together these form
the up-down cycle. The partial map up $\upsilon_0$ maps $2' \mapsto 3$
(curved solid arrow) and the partial map $\upsilon_1$ maps $1' \mapsto 2,
3' \mapsto 3$ (curved dashed arrows). The remaining arrows are fragments
of paths leading from the set of tops to the up-down cycle and from the
up-down cycle to $A \cup B$.

The map $\pi_0$ constructed from $(\delta_0,\upsilon_0)$ as in
Definition~\ref{def:sgndeltaupsilon} can be read off from the left-hand
picture: it maps $1 \mapsto 1',2 \mapsto 3'$ and has sign $1$; the
number of cycles in the digraph $H_0$ as in that definition is zero. So
$\sgn(\delta_0,\upsilon_0)=1$.  Similarly, $\pi_1$ constructed from
$(\delta_1,\upsilon_1)$ is $1 \mapsto 2'$, with sign $1$, and $H_1$
has two cycles $21'$ and $33'$, so $\sgn(\delta_1,\upsilon_1)=1$, as well.

On the right, the roles of $\delta_0$ and $\delta_1$ are reversed. The
pair $(\delta_1,\upsilon_0)$ leads to no cycles and to the map $1 \mapsto 3', 2 \mapsto 1'$
with sign $-1$. Thus $\sgn(\delta_1,\upsilon_0)=-1$.
Similarly, $(\delta_0,\upsilon_1)$ leads to the map $1 \mapsto 2'$ with
sign $1$ and to one cycle $33'$, so $\sgn(\delta_0,\upsilon_1)=-1$. Note that
$\sgn(\delta_0,\upsilon_0)
\sgn(\delta_1,\upsilon_1)=
\sgn(\delta_1,\upsilon_0)
\sgn(\delta_0,\upsilon_1)$
in accordance with Lemma~\ref{lm:delta01}, and this shows that the sign
of $\bfT$ does not change under swapping the $0$ and $1$-labels in the
up-down cycle, resulting in $\bfT'$.
\end{ex}


\bibliographystyle{plain}

\begin{thebibliography}{99}

\bibitem{Drton2010}  M.~Drton, R.~Foygel, and S.~Sullivant.  Global identifiability of linear structural equation models, \emph{Annals of Statistics} {\bf 39} (2011):865-886

\bibitem{Fomin2001} S.~Fomin. Loop-erased walks and total positivity,
\emph{Trans. Amer. Math. Soc.} {\bf 353(9)} (2001): 3563--3583

\bibitem{Foygel2011} R.~Foygel, M.~Drton, J.~Draisma. Half-trek criterion
for generic identifiability of linear structural equation models,
\emph{Annals of Statistics}, to appear, 2012; preprint version {\tt
arXiv:1107.5552}

\bibitem{Gessel1985} I.~Gessel and G.~Viennot.  Binomial determinants, paths, and hook length
formulae. \emph{Adv. Math.} {\bf 58} (1985) 300--321.

\bibitem{Lalonde1996} P.~Lalonde.  A non-commutative version of Jacobi's equality on the cofactors of a matrix. \emph{Discrete Mathematics}, {\bf 158} (1996) 161--172.

\bibitem{Lauritzen1996} S.~Lauritzen. \emph{Graphical Models}, Oxford University Press, 1996.

\bibitem{Lindstrom1973} B.~Lindstr\"om. On the vector representations of induced matroids. \emph{Bull.
Lond. Math. Soc.} {\bf 5} (1973) 85--90.

\bibitem{Pearl2000}  J.~Pearl.  \emph{Causality:  Models, Reasoning, and Inference}, Cambridge University Press, 2000.

\bibitem{Sullivant2010}
S.~Sullivant, K.~Talaska, and J.~Draisma.
Trek separation for Gaussian graphical models. \emph{ Annals of Statistics} {\bf 38} no.3 (2010) 1665--1685

\bibitem{Talaska12}
K.~Talaska. Determinants of weighted path matrices. \emph{
Preprint}, available from \verb+http://arxiv.org/abs/1202.3128+

\end{thebibliography}

\end{document}